\documentclass[a4,12pt]{amsart}
\usepackage{amssymb,amstext,amsmath,amscd,amsthm,amsfonts,enumerate,latexsym}
\usepackage{color}
\usepackage[dvipdfmx]{graphicx}
\usepackage[all]{xy}
\theoremstyle{plain}
\newtheorem{thm}{Theorem}[section]

\newtheorem*{thm*}{Theorem}
\newtheorem*{cor*}{Corollary}

\newtheorem{proposition}[thm]{Proposition}
\newtheorem{lemma}[thm]{Lemma}

\newtheorem{cor}[thm]{Corollary}

\newtheorem*{claim*}{Claim}

\theoremstyle{definition}

\newtheorem{definition}[thm]{Definition}
\newtheorem{ex}[thm]{Example}

\newtheorem{fact}[thm]{Fact}

\theoremstyle{remark}

\numberwithin{equation}{thm}

\def\Ker{\operatorname{Ker}}

\def\mod{\mathrm{mod}}

\def\m{\mathfrak m}

\newcommand{\rme}{\mathrm{e}}

\newcommand{\rmr}{\mathrm{r}}

\newcommand{\rmK}{\mathrm{K}}

\newcommand{\rmQ}{\mathrm{Q}}

\newcommand{\fka}{\mathfrak{a}}
\newcommand{\fkb}{\mathfrak{b}}
\newcommand{\fkc}{\mathfrak{c}}

\newcommand{\fkm}{\mathfrak{m}}

\newcommand{\fkp}{\mathfrak{p}}
\newcommand{\fkq}{\mathfrak{q}}

\newcommand{\mapright}[1]{%
\smash{\mathop{%
\hbox to 1cm{\rightarrowfill}}\limits^{#1}}}

\newcommand{\mapleft}[1]{%
\smash{\mathop{%
\hbox to 1cm{\leftarrowfill}}\limits_{#1}}}

\def\depth{\operatorname{depth}}
\def\AGL{\operatorname{AGL}}

\def\Ass{\operatorname{Ass}}

\def\GGL{\operatorname{GGL}}
\def\AGL{\operatorname{AGL}}
\def\gr{\mbox{\rm gr}}

\title[Characterization of GGL rings]{Characterization of Generalized Gorenstein rings}

\author{Shiro Goto}
\address{Department of Mathematics, School of Science and Technology, Meiji University, 1-1-1 Higashi-mita, Tama-ku, Kawasaki 214-8571, Japan}
\email{shirogoto@gmail.com}

\author{Ryotaro Isobe}
\address{Department of Mathematics and Informatics, Graduate School of Science and Technology, Chiba University, Chiba-shi 263, Japan}
\email{r.isobe.math@gmail.com}

\author{Shinya Kumashiro}
\address{Department of Mathematics and Informatics, Graduate School of Science and Technology, Chiba University, Chiba-shi 263, Japan}
\email{polar1412@gmail.com}

\author{Naoki Taniguchi}
\address{Department of Mathematics, School of Science and Technology, Meiji University, 1-1-1 Higashi-mita, Tama-ku, Kawasaki 214-8571, Japan}
\email{taniguchi@meiji.ac.jp}
\urladdr{http://www.isc.meiji.ac.jp/~taniguchi/}


\thanks{The first author was partially supported by JSPS Grant-in-Aid for Scientific Research (C) 16K05112. The fourth author was partially supported by JSPS Grant-in-Aid for Young Scientists (B) 17K14176.}

\begin{document}
\maketitle

\setlength{\baselineskip}{15pt}

\begin{abstract}
The notion of generalized Gorenstein local ring ($\GGL$ ring for short) is one of the generalizations of Gorenstein rings. In this article, there is given a characterization of $\GGL$ rings in terms of their canonical ideals and related invariants.
\end{abstract}




\section{Introduction}

The notion of a generalized Gorenstein local ring ($\GGL$ ring for short) is one of the generalizations of Gorenstein rings. Similarly, for almost Gorenstein local rings ($\AGL$ rings for short), the notion is given in terms of a certain specific embedding of the rings into their canonical modules (see \cite{GK} by the first and third authors). However, the research on $\AGL$ rings developed in \cite{GMP} by the first author, N. Matsuoka, and T. T. Phuong for arbitrary one-dimensional Cohen--Macaulay local rings is based on an investigation of the relationship between two invariants: the first Hilbert coefficient of the canonical ideals and the Cohen--Macaulay type of the rings. Therefore, it seems natural to ask for a possible characterization of AGL rings of higher dimension and that of GGL rings in terms of their canonical ideals and some related invariants. For AGL rings, this characterization has been accomplished by the first and fourth authors and R. Takahashi. They have already given a satisfactory result \cite[Theorem 5.1]{GTT}. The present purpose is to perform the task for GGL rings of higher dimension.

Originally, the series of studies in \cite{CGKM, GGHV, GK, GK2, GMP, GMTY1, GMTY2, GMTY3, GMTY4, GRTT, GTT, GTT2, GT, T} aimed to find a new class of Cohen--Macaulay local rings, which contains the class of Gorenstein rings. AGL rings are one of the candidates for such a class. Historically, the notion of an AGL ring in our sense has its roots in \cite{BF} by V. Barucci and R. Fr\"oberg in 1997, where they dealt with one-dimensional analytically unramified local rings. They also explored numerical semigroup rings, starting a very beautiful theory. In \cite{GMP}, the authors relaxed the notion to arbitrary Cohen--Macaulay local rings of dimension one based on a different point of view. Repairing a gap in the proof of \cite[Proposition 25]{BF}, they provided new possible areas of study for one-dimensional Cohen--Macaulay local rings. Among various results in \cite{GMP}, the most striking achievement seems that their arguments have prepared for a possible definition \cite[Definition 3.3]{GTT} of AGL rings of higher dimension. We now have two more notions: the $2$-almost Gorenstein local ring (\cite{CGKM}) and GGL ring (\cite{GK}), both of which are reasonable candidate generalizations of Gorenstein rings and almost Gorenstein rings as well.

As stated above, the present purpose is to find a characterization of GGL rings in terms of the canonical ideals and related invariants. To state our motivation and the results more precisely, let us review the definition of GGL rings. Throughout this article, let $(R,\m)$ be a Cohen--Macaulay local ring with $d= \dim R > 0$, possessing the canonical module $\rmK_R$. For simplicity, let us always assume that the residue class field $R/\m$ of $R$ is infinite. Let $\fka$ be an $\m$-primary ideal of $R$. With this notation, the definition of a $\GGL$ ring is stated as follows.

\begin{definition}[\cite{GK}]\label{def1.1}
We say that $R$ is a generalized Gorenstein local ring ($\GGL$ for short) if one of the following conditions is satisfied:
\begin{enumerate}[{\rm (1)}]
\item $R$ is a Gorenstein ring.
\item $R$ is not a Gorenstein ring, but there exists an exact sequence
$$0 \to R \xrightarrow{\varphi} \rmK_R \to C \to 0$$
of $R$-modules such that
\begin{enumerate}
\item[$\mathrm{(i)}$] $C$ is an Ulrich $R$-module with respect to $\fka$, and
\item[$\mathrm{(ii)}$] the induced homomorphism $R/\fka \otimes_R \varphi : R/\fka \to \rmK_R/\fka \rmK_R$ is injective.
\end{enumerate}
\end{enumerate}
When  Case (2) occurs, we especially say that $R$ is a $\GGL$ ring with respect to $\fka$.
\end{definition}

Let us further explain Definition \ref{def1.1}. Let $M$ be a finitely generated $R$-module of dimension $s\ge0$. Then, we say that $M$ is an Ulrich $R$-module with respect to $\fka$ if the following three conditions are satisfied:
\begin{enumerate}[{\rm (i)}]
\item $M$ is a Cohen--Macaulay $R$-module,
\item $\rme_\fka^0(M) = \ell_R(M/\fka M)$, and
\item $M/\fka M$ is a free $R/\fka$-module,
\end{enumerate}
where $\ell_R(*)$ is the length, and
$$\rme_\fka^0(M) = \lim_{n\to \infty}s!{\cdot}\frac{\ell_R(M/\fka^{n+1}M)}{n^s}$$
denotes the multiplicity of $M$ with respect to $\fka$ (\cite{GOTWY}). The notion of an Ulrich $R$-module with respect to an $\m$-primary ideal is a generalization of a maximally generated maximal Cohen--Macaulay $R$-module (i.e., a maximal Ulrich $R$-module with respect to $\fkm$; see \cite{BHU}). One can consult \cite{GK, GOTWY, GOTWY2, GTT} for the basic properties of Ulrich modules in our sense. Here, let us note one thing. In the setting of Definition \ref{def1.1}, suppose that there is an exact sequence $$0 \to R \to \rmK_R \to C \to 0$$ of $R$-modules such that $C \ne (0)$. Then, $C$ is a Cohen--Macaulay $R$-module of dimension $d-1$ (\cite[Lemma 3.1 (2)]{GTT}), and $C$ is an Ulrich $R$-module with respect $\fka$ if and only if $C/\fka C$ is a free $R/\fka$-module and
$$\fka C = (f_2, f_3, \ldots, f_{d})C$$
for some elements $f_2, f_3, \ldots, f_d \in \fka$ (\cite[Proof of Proposition 2.4, Claim]{GK}). Therefore, if $\fka = \m$, Definition \ref{def1.1} is exactly the same as that of AGL rings given by \cite[Definition 3.3]{GTT}. In \cite{GK}, the authors investigate GGL rings, and one can find a report of the basic results for GGL rings, which greatly generalizes several results in \cite{GTT}, clarifying what AGL rings are.

The present purpose is to give a characterization of GGL rings. Let $\rmr(R)$ stand for the Cohen--Macaulay type of $R$. We then have the following, which is the main result of this article.

\begin{thm}\label{1.2} Let $(R,\m)$ be a Cohen--Macaulay local ring with $d=\dim R >0$ and an infinite residue class field, possessing the canonical module $\rmK_R$. Let $I \subsetneq R$ be an ideal of $R$ such that $I \cong \rmK_R$ as an $R$-module. Let $\fka$ be an $\m$-primary ideal of $R$. Assume that $R$ is not a Gorenstein ring. Then, the following conditions are equivalent.
\begin{enumerate}[{\rm (1)}]
\item $R$ is a $\GGL$ ring with respect to $\fka$.
\item There exists a parameter ideal $Q = (f_1, f_2, \ldots, f_d)$ of $R$ such that $f_1 \in  I$ and, setting $J = I + Q$, the following three conditions are satisfied.
\begin{enumerate}[{\rm (i)}]
\item $\fka = Q:_RJ$.
\item $\fka J = \fka Q$.
\item $\rme_J^1(R) = \ell_R(R/\fka){\cdot}\rmr(R)$.
\end{enumerate}

\noindent
When this is the case, $R/\fka$ is a Gorenstein ring, and the following assertions hold true.
\end{enumerate}
\begin{enumerate}[{\rm (a)}]
\item $J^3 = QJ^2$, but $J^2 \ne QJ$.
\item $\mathcal{S}_Q(J) \cong (\mathcal{T}/\fka \mathcal{T})(-1)$, as a graded $\mathcal{T}$-module, where $\mathcal{S}_Q(J)$ $($resp. $\mathcal{T} = \mathcal{R}(Q)$$)$ denotes the Sally module of $J$ with respect to $Q$ $($see \cite{V}$)$ $($resp. the Rees algebra of $Q$$)$.
\item $f_2, f_3, \ldots, f_d$ form a super-regular sequence of $R$ with respect to $J$, and $\depth \operatorname{gr}_J(R) = d-1$, where $\operatorname{gr}_J(R) = \bigoplus_{n \ge 0}J^{n}/J^{n+1}$ denotes the associated graded ring of $J$.
\item The Hilbert function of $R$ with respect to $J$ is given by
{\footnotesize $$\ell_R(R/J^{n+1}) =
\rme_J^0(R){\cdot}\binom{n+d}{d} -\left[\rme_J^0(R) - \ell_R(R/J) + \ell_R(R/\fka)\right]{\cdot}\binom{n+d-1}{d-1} + \ell_R(R/\fka){\cdot}\binom{n+d-2}{d-2}
$$}
\noindent
for $n \ge 1$. Hence, $\rme_J^2(R) = \ell_R(R/\fka)$ if $d \ge 2$, and $\rme_J^i(R) = 0$ for all $3 \le i \le d$ if $d \ge 3$.
\end{enumerate}
\end{thm}

The study of $\GGL$ rings is still in progress, and Theorem \ref{1.2} now completely generalizes the corresponding assertion \cite[Theorem 5.1]{GTT} of $\AGL$ rings to arbitrary $\GGL$ rings of higher dimension, certifying that the notion of a $\GGL$ ring is a reasonable generalization of $\AGL$ rings and that a $\GGL$ ring is one of the candidates of broader notion that generalizes Gorenstein rings.

We now briefly explain how this paper is organized. The proof of Theorem \ref{1.2} shall be given in Sections 3 and 4. In Section 2, we summarize some of the known results given by \cite{GK}, which we use to prove Theorem \ref{1.2}. We will explore an example in order to illustrate Theorem \ref{1.2} in Section 5.


\section{Preliminaries}\label{section2}


In this section we summarize some of the results in \cite[Section 4]{GK} about $\GGL$ rings , which we use to prove Theorem \ref{1.2}.

First, let $(R,\m)$ be a Cohen-Macaulay local ring of dimension one, admitting the canonical module $\rmK_R$. Let $I \subsetneq R$ be an ideal of $R$ such that $I \cong \rmK_R$ as an $R$-module. We assume that $I$ contains a parameter ideal $(a)$ of $R$ as a reduction.
We set $$K = \frac{I}{a} = \left\{\frac{x}{a} ~\middle|~ x \in I\right\}$$ in the total ring $\rmQ(R)$ of fractions of $R$. Hence $K$ is a fractional ideal of $R$ such that $R \subseteq K \subseteq \overline{R}$ and $K\cong \rmK_R$, where $\overline{R}$ denotes  the integral closure of $R$ in $\rmQ(R)$. We set $S = R[K]$ in $\rmQ(R)$. Hence $S$ is a module-finite birational extension of $R$. Note that the ring $S=R[K]$ is independent of the choice of canonical ideals $I$ and reductions $(a)$ of $I$ (\cite[Theorem 2.5]{CGKM}). We set $\fkc=R:S$. We then have the following, which shows the $\fkm$-primary ideal $\fka$ which appears in Definition \ref{def1.1} of a $\GGL$ ring $R$ is uniquley determined, when $\dim R = 1$. We note a brief proof.

\begin{proposition}[{\cite{GK}}]\label{2.1}
Suppose $R$ is not  a Gorenstein ring but $R$ is a $\GGL$ ring with respect to an $\fkm$-primary ideal $\fka$ of $R$. Then $\fka = \fkc$.
\end{proposition}

\begin{proof}
We choose an exact sequence
$$
0 \to R \xrightarrow{\varphi} I \to C \to 0
$$
of $R$-modules such that $C$ is an Ulrich $R$-module with respect to $\fka$ and the induced homomorphism $R/\fka \otimes_R \varphi : R/\fka \to I/\fka I$ is injective. We set $f = \varphi(1)$ and identify $C = I/(f)$.  Then $\fka I \subseteq (f)$ since $\fka {\cdot}(I/(f)) = (0)$, while $(f) \cap \fka I =\fka f$ since the homomorphism $R/\fka \otimes_R \varphi$ is injective. Consequently, $\fka I = \fka f$, whence $(f)$ is a reduction of $I$.  We consider $L = \frac{I}{f}$ and set $S = R[L]$. Then $\fka L = \fka$ since $\fka I = \fka f$, so that $\fka S = \fka$ since $S = L^n$ for all $n \gg 0$. Therefore, $\fka \subseteq \fkc = R:S$, so that $\fka = \fkc$, because $\fkc \subseteq R:L = (f):_R I = \fka$.
\end{proof}

In general we have the following.

\begin{fact}[\cite{GK}]\label{2.2}
Let $\fka=R:K$. Then the following conditions are equivalent.
\begin{enumerate} [$(1) $]
\item $K^2=K^3$.
\item $\fka=\fkc$.
\item $\fka K=\fka$.
\end{enumerate}
\end{fact}

The key in the theory of one-dimensional $\GGL$ rings is the following, which we shall freely use in the present article. See \cite[Section 4]{GK} for the proof.

\begin{thm}[\cite{GK}]\label{2.3}
Suppose that $R$ is not a Gorenstein ring. Then the following conditions are equivalent.
\begin{enumerate}[$(1)$]
\item $R$ is a $\GGL$ ring $($necessarily with respect to $\fkc$$)$.
\item $K/R$ is a free $R/\fkc$-module.
\item $K/\fkc=K/\fkc K$ is a free $R/\fkc$-module.
\item $S/R$ is a free $R/\fkc$-module.
\item $S/\fkc=S/\fkc S$ is a free $R/\fkc$-module.
\item $\rme_1(I)=\ell_R(R/\fkc){\cdot}\rmr(R)$.
\end{enumerate}
When this is the case,  the following assertions hold true.
\begin{enumerate}[{\rm (i)}]
\item $K^2=K^3$.
\item $R/\fkc$ is a Gorenstein ring.
\item $S/K\cong R/\fkc$.
\end{enumerate}
\end{thm}


Let us note here the non-zerodivisor characterization given by \cite{GK} for $\GGL$ rings of higher dimension. To state it, let $(R,\m)$ be a Cohen-Macaulay local ring with $d=\dim R >0$ and infinite residue class field, possessing the canonical module $\rmK_R$. We then have the following.

\begin{thm}[\cite{GK}]\label{2.4}
Suppose that $R$ is not a Gorenstein ring.
Then the following assertions hold true.
\begin{enumerate}
\item[$(1)$] Let $R$ be a $\GGL$ ring with respect to $\fka$ and $d \ge 2$. Consider the exact sequence
$$
0 \to R \xrightarrow{\varphi} \rmK_R \to C \to 0
$$
of $R$-modules such that $C$ is an Ulrich $R$-module with respect to $\fka$ and the induced homomorphism $R/\fka \otimes_R\varphi : R/\fka \to \rmK_R/\fka \rmK_R$ is injective. Choose a superficial element $f\in \fka$ for $C$ with respect to $\fka$ and assume that $f$ is $R$-regular. Then $R/(f)$ is a $\GGL$ ring with respect to $\fka/(f)$, possessing
$$
0 \to R/(f) \xrightarrow{R/(f) \otimes_R \varphi} \rmK_R/f\rmK_R \to C/fC \to 0
$$ to be a defining exact sequence.
\item[$(2)$] Let $f\in \fkm$ be $R$-regular and assume that $R/(f)$ is a $\GGL$ ring with respect to $[\fka+(f)]/(f)$. Then $R$ is a $\GGL$ ring with respect to $\fka+(f)$ and $f\not\in \fkm [\fka + (f)]$.
\end{enumerate}
\end{thm}


\section{Proof of the main part of Theorem \ref{1.2}}
Let $(R,\m)$ be a Cohen-Macaulay local ring with $d=\dim R >0$ and infinite residue class field, possessing the canonical module $\rmK_R$. Let $I \subsetneq R$ be an ideal of $R$ such that $I \cong \rmK_R$ as an $R$-module. We choose a parameter ideal $Q = (f_1, f_2, \ldots, f_d)$ of $R$ so that $f_1 \in  I$. We set $\fkq=(f_2, f_3, \ldots, f_d)$ and $J = I + \fkq$. Let $\fka$ be an $\m$-primary ideal of $R$. In Sections 3 and 4 we throughout assume that $R$ is not a Gorenstein ring. The purpose is to prove the equivalence between Conditions (1) and (2) in Theorem \ref {1.2}.

Let us begin with the following.

\begin{proposition}\label{3.1}
$\fkq \cap I = \fkq I$ and $J \ne Q$.
\end{proposition}

\begin{proof}
We get $\fkq \cap I = \fkq I$, since $\fkq$ is a parameter ideal of the Gorenstein ring $R/I$. If $J=Q$, then $I = Q \cap I = (f_1) + (\fkq \cap I)$, whence $I = (f_1)$ by the first equlaity. Therefore, $R$ is a Gorenstein ring, which is impossible.
\end{proof}

\begin{thm}\label{3.2}
The following conditions are equivalent.
\begin{enumerate}[{\rm (1)}]
\item $\fka = Q :_RJ$, $\fka J = \fka Q$, and $\rme_J^1(R)=\ell_R(R/\fka){\cdot}\rmr(R)$.
\item $\fka J = \fka Q$ and the $R/\fka$-module $J/Q$ is free.
\end{enumerate}
When this is the case, $R/\fkq$ is a $\GGL$ ring with respect to $\fka/\fkq$, whence so is the ring $R$ with respect to $\fka$.
\end{thm}

\begin{proof} We may assume $\fka J = \fka Q$. Hence, $Q$ is a reduction of $J$. Because  $J/Q \ne (0)$ by Proposition  \ref{3.1}, we get $\fka=Q:_RJ$, once $J/Q$ is $R/\fka$-free. Consequently, we may also assume that $\fka = Q:_RJ$. First, suppose that $d=1$. Hence $J = I$. We set $K = \frac{I}{f_1}$ in the total ring of fractions of $R$. Then, since $\fka K = \fka$, we have $\fka = R : R[K]$ by Fact \ref{2.2}. Consequently, by Theorem \ref{2.3}, $R$ is a $\GGL$ ring (necessarily with respect to $\fkc$; see Proposition \ref{2.1}) if and only if $\rme_I^1(R) = \ell_R(R/\fka){\cdot}\rmr(R)$. By Theorem \ref{2.3}, the former condition is also equivalent to saying that $I/Q ~(\cong K/R)$ is a free $R/\fka$-module, whence the equivalence of Conditions (1) and (2) follows.

Let us consider the case where $d \ge 2$. Assume that Condition (2) is satisfied. Let us check that $\overline{R}=R/\fkq$ is a $\GGL$ ring.  Set $$\overline{R} = R/\fkq, \ \  \overline{Q}= Q/\fkq, \ \ \overline{J}= J/\fkq,\ \ \text{and}\ \ \overline{\fka}= \fka/\fkq.
$$ We then have $\overline{J}=(I + \fkq)/\fkq \cong I/\fkq I = \rmK_{\overline{R}}$, since $I \cong \rmK_R$ and $f_2, f_3, \ldots, f_d$ is a regular sequence for the $R$-module $I$. Consequently, because $\overline{\fka}{\cdot}\overline{J}=\overline{\fka}{\cdot}\overline{Q}$ and $\overline{J}/\overline{Q}$ is $\overline{R}/\overline{\fka}$-free, from the case of $d=1$ it follows that $\overline{R}$ is a $\GGL$ ring (Fact \ref{2.2} and Theorem \ref{2.3}), whence so is $R$ with respect $\fka$ (Theorem \ref{2.4} (2)).

We now assume that the implication (2) $\Rightarrow$ (1) holds true for $d-1$. Since $Q$ is a reduction of $J$ and the field $R/\fkm$ is infinite, there exist elements $h_1, h_2, \ldots, h_d \in Q$ such that $\mathrm{(i)}$ $h_1 \in I$, $\mathrm{(ii)}$ $Q=(h_1, h_2, \ldots, h_d)$, and $\mathrm{(iii)}$ $h_2$ is superficial for $R$ with respect to $J$. This time, we consider the ring $\overline{R} = R/(h_2)$ and let $\overline{*}$ denote the reduction $\mod~(h_2)$. Then $\overline{I} = [I+(h_2)]/(h_2) \cong I/h_2I = \rmK_{\overline{R}}$ and $\overline{h_1} \in \overline{I}$. Condition (2) is clearly satisfied for the ring $\overline{R}$ as for the ideals $\overline{\fka}$, $\overline{Q}$, and $\overline{J}$. Therefore,  by the hypothesis of induction on $d$ we get $$\rme_{\overline{J}}^1(\overline{R})=\ell_{\overline{R}}\left(\overline{R}/\overline{\fka}\right){\cdot}\rmr(\overline{R}).$$ Consequently, $\rme_J^1(R) = \ell_R(R/\fka){\cdot}\rmr(R)$, because $\rme_J^1(R) = \rme_{\overline{J}}^1(\overline{R})$ (remember that $h_2$ is superficial for $R$ with respect to $J$).

The reverse implication (1) $\Rightarrow$ (2) also follows by induction on $d$, chasing the above argument in the opposite direction.
\end{proof}

Let us suppose, with the same notation as in Theorem \ref{3.2}, that the equivalent conditions of Theorem \ref{3.2} are satisfied. Then we get the following.

\begin{proposition}\label{3.3}
$J \subseteq \fka$. Hence $J^2 \subseteq Q$.
\end{proposition}

\begin{proof}
Suppose $d=1$. We set $K = \frac{I}{f_1}$, $S = R[K]$, and $\fkc = R:S$. Then since $R$ is a $\GGL$ ring with respect $\fka$, we get $\fka = \fkc$ by Proposition \ref{2.1} and $\fkc = R:K$ (see Fact \ref{2.2} and Theorem \ref{2.3}). Therefore, $I \subseteq \fkc$, because  $\fkc$ is an ideal of $S$ and $f_1 \in \fkc$ (note that $I= f_1 K \subseteq R$). If $d > 1$, then passing to $R/\fkq$, we have $J/\fkq \subseteq \fka/\fkq$, whence $J \subseteq \fka$. Therefore $J^2 \subseteq Q$, because $\fka = Q:_RJ$.
\end{proof}

We are now ready to prove the equivalence of Conditions (1) and (2) in Theorem \ref{1.2}.

\begin{proof}[Proof of the main part in Theorem {\rm \ref{1.2}}]
See Theorem \ref{3.2} for the implication (2) $\Rightarrow$ (1). To see the implication (1) $\Rightarrow$ (2), we consider the exact sequence
$$(\sharp) \ \ \ 0 \to R \xrightarrow{\varphi} I \to C \to 0$$
of $R$-modules such that $C$ is an Ulrich $R$-module with respect to $\fka$ and the induced homomorphism $R/\fka \otimes_R \varphi : R/\fka \to I/\fka I$ is injective. Let $f_1 = \varphi(1) \in I$. Then $f_1$ is a non-zerodivisor of $R$. Choose elements $f_2, f_3, \ldots, f_d \in \fka$ so that in the ring $R'=R/(f_1)$ these elements  generate a reduction of $\fka{\cdot} R'$. Then $f_1, f_2, \ldots, f_d$ is a system of parameters of $R$ with $f_1 \in I$. Set $\fkq = (f_2, f_3, \ldots, f_d)$, $Q = (f_1)+ \fkq$, and $J = I + \fkq$. Then, $\fkq$ is a parameter ideal of $R/I$, and because $$\ell_{R'}(C/\fkq C)=\rme_{\fka {\cdot} R'}^0(C) = \rme_{\fka}^0(C) = \ell_R(C/\fka C),$$
 we get $\fka C = \fkq C$. Therefore, since $\fka^n C = \fkq^n C$ for all $n \in \Bbb Z$, $f_2, f_3, \ldots, f_d$ forms a super-regular sequence for $C$ with respect to $\fka$, whence it is a superficial sequence for $C$ with respect to $\fka$. We set $\overline{R} = R/\fkq$, $\overline{I} = I/\fkq I$, and $\overline{C} = C/\fkq C$. Then, from exact sequence ($\sharp$) we get the exact sequence
$$
(\overline{\sharp}) \ \ \ 0 \to \overline{R} \xrightarrow{\overline{\varphi}} \overline{I} \to \overline{C} \to 0
$$
of $\overline{R}$-modules. By Theorem \ref{2.4} (1) the ring $\overline{R}$ is a $\GGL$ ring with respect to $\fka/\fkq$, possessing sequence $(\overline{\sharp})$ to be a defining exact sequence. Consequently, because
$$
J/\fkq =(I+\fkq)/\fkq \cong \overline{I} = \rmK_{\overline{R}},
$$
by Theorem \ref{2.3} it follows that $\overline{\fka}{\cdot}\overline{J} = \overline{\fka}{\cdot}\overline{Q}$ and $\overline{J}/\overline{Q}$ is $\overline{R}/\overline{\fka}$-free, where
$$\overline{\fka}=\fka/\fkq, \ \overline{J} = J/\fkq,\ \text{and}\ \overline{Q} = Q/\fkq.$$
Hence $J/Q$ is $R/\fka$-free and $\fka J \subseteq \fka Q + \fkq$. Therefore $$\fka J = (\fka Q + \fkq) \cap \fka J = \fka Q + [\fkq \cap \fka J] =\fka Q + [\fkq \cap \fka I] \subseteq \fka Q + [\fkq \cap  I] = \fka Q +\fkq I,
$$
where the third equality follows from the fact that $\fka J = \fka I +\fka \fkq$. Hence $\fka J = \fka Q$, because $I \subseteq \fka$  by Proposition \ref{3.3}. Thus Theorem \ref{3.2} certifies that Condition (2) in Theorem \ref{1.2} is satisfied for the ideals $\fka$, $Q$, and $J$. This completes the proof of the equivalence of Conditions (1) and (2) in Theorem \ref{1.2}.
\end{proof}

\section{Proof of the last assertions of Theorem \ref{1.2}}
Let us show the last assertions of Theorem \ref{1.2}. In what follows, assume that our ideals  $Q$ and $J$ satisfy the equivalent conditions in Theorem \ref{3.2}. Hence $R$ (resp. $R/\fkq$) is a $\GGL$ ring with respect to $\fka$ (resp. $\fka/\fkq$), and $R/\fka$ is a Gorenstein ring by Theorem \ref{2.3} (ii). To prove the last assertions of Theorem \ref{1.2}, we need some preliminaries. Let us  maintain the same notation as in the proof of Theorem \ref{3.2}.

We begin with the following.

\begin{lemma}\label{3.0}
$\fkq \cap J^2 = \fkq J$.
\end{lemma}

\begin{proof}
Remember that $J^2 = \fkq J + I^2$, since $J = I + \fkq$. We then have $\fkq \cap J^2 = \fkq J + (\fkq \cap I^2)$, so that $\fkq \cap J^2 = \fkq J$, because $\fkq \cap I^2 \subseteq \fkq \cap I = \fkq I$.
\end{proof}

\begin{proposition}\label{3.4} We set $L = Q:_R\fka$. Then the following assertions hold true.
\begin{enumerate}[{\rm (1)}]
\item $L = J:_R\fka$ and $L^2 \subseteq Q$.
\item $L/J\cong R/\fka$ as an $R$-module.
\item $J^2/QJ \cong R/\fka$ as an $R$-module.
\item $\fka L = \fka Q$.
\item $L^2 = QL$.
\item $J^3=QJ^2$ but $J^2 \ne QJ$.
\end{enumerate}
\end{proposition}

\begin{proof}
(1), (2), (3)~First, consider the case where $d = 1$. Let us maintain the notation of the proof of Proposition \ref{3.3}. Then $Q :_R \fkc = I :_R\fkc = f_1S$. In fact, we have $\fkc = K:S=R:K$ (see Fact \ref{2.2} and Theorem \ref{2.3}) and hence $f_1 \in \fkc$ . Let $x \in R$. Then $x{\cdot}\fkc \subseteq I$ if and only if $\frac{x}{f_1}{\cdot}\fkc \subseteq K$. The latter condition is equivalent to saying that $\frac{x}{f_1} \in S$, since $K : \fkc = K:(K:S)=S$. Thus $I:_R\fkc = f_1S$. Because $f_1S{\cdot}\fkc = f_1\fkc \subseteq Q=(f_1),$ we get $I:_R\fkc = f_1S \subseteq Q :_R\fkc$. Hence $Q:_R\fkc = I :_R\fkc = f_1S$. Consequently $$(Q:_R\fkc)^2 = f_1(f_1S) \subseteq Q=(f_1),$$ and $[Q:_R\fkc]/I = f_1S/f_1K \cong S/K \cong R/\fkc$ by Theorem \ref{2.3}, which proves Assertions (1) and (2), because $\fka = \fkc$.
Assertion (3) is now clear, since $$I^2/f_1I \cong K^2/K \cong R/\fkc$$ by Theorem \ref{2.3}. Now consider the case where $d \ge 2$. To show Assertions (1) and (2), passing to the ring $R/\fkq$, we can safely assume that $d=1$, and  we have already done with the case. Consider Assertion (3). We set $\overline{R} = R/\fkq$ and denote by $\overline{*}$ the reduction $\mod~\fkq$. Let $\varphi : J^2/QJ \to \overline{J}^2/f_1\overline{J}$ be the canonical epimorphism. We then have $\Ker \varphi = [J^2 \cap (\fkq + f_1J)]/QJ$. Hence, because $J^2 \cap \fkq = \fkq J$ by Lemma \ref{3.0}, we have
$$J^2 \cap (\fkq + f_1J)= f_1J + (J^2 \cap \fkq)=f_1J + \fkq J =QJ,$$
whence the required isomorphism $J^2/QJ \cong R/\fka$ follows.

(4)~Suppose $d=1$. Then $\fkc L = \fkc{\cdot}f_1S =f_1{\cdot}\fkc$, whence the assertion follows. Suppose that $d \ge 2$ and that Assertion (4) holds true for $d-1$. Note that $Q = (f_1, f_1 + f_2)+(f_3, \ldots, f_d)$. Then, because $R/(f_1+f_2)$ and $R/(f_2)$ are $\GGL$ rings with respect to $\fka/(f_1+f_2)$ and $\fka/(f_2)$ respectively, thanks to the hypothesis of induction on $d$, we get
$$\fka L \subseteq [\fka Q + (f_1+f_2)] \cap[\fka Q + (f_2)]=\fka Q +\left\{(f_1+f_2) \cap [\fka Q + (f_2)]\right\}.$$
Since $\fka Q + (f_2) \subseteq  \fka {\cdot} (f_1+f_2) + (f_2, f_3, \ldots, f_d)$, we furthermore have that
$$\fka L \subseteq \fka Q + \left\{(f_1+f_2) \cap [\fka {\cdot} (f_1+f_2) + (f_2, f_3, \ldots,f_d)] \right\}= \fka Q + (f_1+f_2){\cdot}(f_2, f_3, \ldots, f_d) = \fka Q.$$
Hence $\fka L = \fka Q$.

(5)~Let $x \in L^2$. Then, $x \in Q$, since $L^2 \subseteq Q$ by Assertion (1). We write $x = \sum_{i=1}^df_ix_i$ with $x_i \in R$. Let $\alpha \in \fka$. Then, because $$\alpha x = \sum_{i=1}^df_i (\alpha x_i) \in \fka L^2 \subseteq Q^2$$ by Assertion (4), we get $\alpha x_i \in Q$ for all $1 \le i \le d$, whence $x_i \in Q :_R\fka =L$. Thus $L^2 = QL$.

(6)~The equality $J^3 = QJ^2$ is a direct consequence of \cite[Proposition 2.6]{GNO2}, since $\mu_R(L/J) = 1$ by Assertion (2). Suppose that $J^2 = QJ$ and let $\overline{*}$ denote the reduction $\mod~\fkq$. Then since $\overline{J} \cong \rmK_{\overline{R}}$ and $\overline{J}^2 = \overline{Q}{\cdot}\overline{J}$,  by \cite[Theorem 3.7]{GMP} $\overline{R}$ is a Gorenstein ring, which is impossible. Hence $J^2 \ne QJ$.
\end{proof}

\begin{proposition}\label{3.5}
The sequence $f_2, f_3, \ldots, f_d$ is a super-regular sequence of $R$ with respect to $J$. Hence $\depth \gr_J(R) = d-1$.
\end{proposition}

\begin{proof}
To see the first assertion, it suffices to show that $\fkq \cap J^{n+1} = \fkq J^n$ for all $n \ge 1$. By Lemma \ref{3.0} we may assume that $n \ge 2$ and that  our assertion holds true for $n-1$. Then, since $J^{n+1}= QJ^n=\fkq J^n + f_1J^n$ by Proposition \ref{3.4} (5), we have
$$\fkq \cap J^{n+1} = \fkq J^n + (\fkq \cap f_1J^n).$$ Consequently, because $\fkq \cap f_1J^n=f_1{\cdot}(\fkq \cap J^n)$ (remember that $f_1, f_2, \ldots, f_d$ is an $R$-regular sequence), by the hypothesis of induction on $n$ we have $\fkq \cap f_1J^n \subseteq \fkq J^n$. Hence $\fkq \cap J^{n+1}= \fkq  J^n$. Consequently, $\depth \gr_J(R) \ge d-1$. Suppose that $\depth \gr_J(R) = d$. Then, $f_1, f_2, \ldots, f_d$ is a super-regular sequence of $R$ with respect to $J$, so that $Q \cap J^2= QJ$. Therefore, $J^2 = QJ$, because $J^2 \subseteq Q$ by Proposition \ref{3.4} (1), which contradicts Proposition \ref{3.4} (6). Hence $\depth \gr_J(R) = d-1$.
\end{proof}

Let $\mathcal{T}= \mathcal{R}(Q)$ and $\mathcal{R}=\mathcal{R}(J)$ be the Rees algebras of $Q$ and $J$ respectively. We now consider the Sally module $\mathcal{S}_Q(J) = J\mathcal{R}/J\mathcal{T}$ of $J$ with respect to $Q$ (see \cite{V}).

\begin{thm}\label{3.6}
$\mathcal{S}_Q(J) \cong (\mathcal{T}/\fka \mathcal{T})(-1)$ as a graded $\mathcal{T}$-module.
\end{thm}

\begin{proof}
We set $\mathcal{S}= \mathcal{S}_Q(J)$ and denote, for each $n \in \Bbb Z$, by $[\mathcal{S}]_n$ the homogeneous component of $\mathcal{S}$ of degree $n$. Then $[\mathcal{S}]_1 = J^2/QJ$ and $\mathcal{S}=\mathcal{T}{\cdot}[\mathcal{S}]_1$ (\cite[Lemma 2.1]{GNO}). Hence by Proposition \ref{3.4} (3), we get an epimorpism $\varphi :(\mathcal{T}/\fka \mathcal{T})(-1) \to \mathcal{S}$ of graded $\mathcal{T}$-modules. Let $X = \Ker \varphi$ and assume that $X\ne (0)$. We choose an element $\fkp \in \Ass_{\mathcal{T}}X$. Then, since $\fkp \in \Ass_{\mathcal{T}}\mathcal{T}/\fka \mathcal{T}$,  and $\mathcal{T}/\fka \mathcal{T}=(R/\fka)[X_1, X_2, \ldots, X_d]$ is the polynomial ring over $R/\fka$ (remember that $\mathcal{T}$ is isomorphic to the symmetric algebra of $Q$ over $R$), we have $\fkp = \m \mathcal{T}$. Then  $\ell_{\mathcal{T}_\fkp}((\mathcal{T}/\fka \mathcal{T})_\fkp) = \ell_R(R/\fka)$, while by \cite[Proposition 2.2]{GNO} we have $$\ell_{\mathcal{T}_\fkp}(\mathcal{S}_\fkp) = \rme_J^1(R) - \ell_R(J/Q).$$ Therefore, $\ell_{\mathcal{T}_\fkp}(\mathcal{S}_\fkp)= \ell_R(R/\fka)$, because $\rme_J^1(R) = \ell_R(R/\fka){\cdot}\rmr (R)$ by Theorem \ref{3.2} and $\ell_R(J/Q) = \ell_R(R/\fka){\cdot}(\rmr(R)-1)$ by Theorem \ref{2.3}. Thus $\ell_{\mathcal{T}_\fkp}((\mathcal{T}/\fka \mathcal{T})_\fkp) =\ell_{\mathcal{T}_\fkp}(\mathcal{S}_\fkp)$, which forces $X_\fkp= (0)$. This is absurd. Thus $(\mathcal{T}/\fka \mathcal{T})(-1) \cong \mathcal{S}$  as a graded $\mathcal{T}$-module.
\end{proof}

Because $$\ell_R(R/J^{n+1}) = \rme_J^0(R){\cdot}\binom{n+d}{d} -\left[\rme_J^0(R) - \ell_R(R/J)\right] {\cdot}\binom{n+d-1}{d-1} - \ell_R([\mathcal{S}]_n)$$ for all $n \in \Bbb Z$ (see \cite[Proposition 2.2]{GNO}), by Theorem \ref{3.6} we readily get the following.

\begin{cor}
The Hilbert function of $R$ with respect to $J$ is given by
{\footnotesize $$\ell_R(R/J^{n+1}) =
\rme_J^0(R){\cdot}\binom{n+d}{d} -\left[\rme_J^0(R) - \ell_R(R/J) + \ell_R(R/\fka)\right]{\cdot}\binom{n+d-1}{d-1} + \ell_R(R/\fka){\cdot}\binom{n+d-2}{d-2}
$$}
\noindent
for $n \ge 1$. Hence, $\rme_J^2(R) = \ell_R(R/\fka)$ if $d \ge 2$, and $\rme_J^i(R) = 0$ for all $3 \le i \le d$ if $d \ge 3$.
\end{cor}

\section{Example}

Let $S=k[[X,Y,Z,V]]$ be the formal power series ring over an infinite field $k$ and let $\fkb = {\Bbb I}_2(\begin{smallmatrix}
X^2&Y+V&Z\\
Y&Z&X^3
\end{smallmatrix})$ denote the ideal of $S$ generated by $2 \times 2$ minors of the matrix $(\begin{smallmatrix}
X^2&Y+V&Z\\
Y&Z&X^3
\end{smallmatrix})$. We set $R = S/\fkb$. We denote by  $x,y,z,v$  the images of $X,Y, Z,V$ in $R$, respectively. Then we have the following.

\begin{ex} The following assertions hold true.
\begin{enumerate}[{\rm (1)}]
\item $R$ is a two-dimensional $\GGL$ ring with respect to $\fka = (x^2,y,z,v)$.\item $\rmr(R) = 2$ and $I=(x^2,y)$ is a canonical ideal of $R$.
\item Set $Q=(x^2,v)$ and $J =I+Q$. Then $Q$ is a parameter ideal of $R$ with $x^2 \in I$.
\item We have $\fka = Q:_RJ$, $\fka J = \fka Q$, and $\rme_J^1(R) = \ell_R(R/\fka){\cdot}\rmr(R)=4$.
\end{enumerate}

\begin{proof}
Since $$R/(v) \cong k[[X,Y,Z]]/{\Bbb I}_2(\begin{smallmatrix}
X^2&Y&Z\\
Y&Z&X^3
\end{smallmatrix}) \cong k[[t^3,t^7,t^8]]$$ where $t$ denotes an indeterminate over $k$, we have $\dim R/(v) =1$. Hence $\operatorname{ht}_S\fkb \ge2$, so that $R$ is a Cohen-Macaulay ring with $\dim R = 2$.  Because $R/vR = k[[t^3,t^7,t^8]]$ is a $\GGL$ ring with respect to $(t^6,t^7,t^8)$ and $v$ is a non-zerodivisor of $R$, by Theorem \ref{2.4} (2) $R$ is a $\GGL$ ring with respect to $\fka$. To see that $I \cong \rmK_R$, note that $(t^6,t^7)$ is a canonical ideal of $k[[t^3,t^7,t^8]]$. Since $R/I = S/(X^2, Y, Z^2)$, the element $v$ acts on $R/I$ as a non-zerodivisor, so that $(v) \cap I = vI$. Hence $J/(v) =[I+(v)]/(v) \cong I/vI$. Because the ideal $J/(v)$ corresponds to $(t^6,t^7)$ under the identification $$R/(v)= k[[X,Y,Z]]/{\Bbb I}_2(\begin{smallmatrix}
X^2&Y&Z\\
Y&Z&X^3
\end{smallmatrix}) = k[[t^3,t^7,t^8]],$$ we see $\rmr_R(I/vI)=1$, where $\rmr_R(*)$ stands for the Cohen-Macaulay type. Therefore $\rmr_R(I)=1$, whence $I \cong \rmK_R$ because $(0):_RI = (0)$. Since $$R/Q \cong k[X,Y,Z,V]/(X^2,Y^2,Z^2,YZ,V),$$ we get $Q:_RJ = Q:_Ry = (x^2,y,z,v) = \fka \supseteq J$. It is direct to check that $\fka J = \fka Q$. The equality  $\rme_J^1(R) = \ell_R(R/\fka){\cdot}\rmr(R)=4$ follows from the fact that $\ell_R(R/(x^2,y,z,v))=\rmr(R)=2$.
\end{proof}

\end{ex}



\end{document}